\documentclass[11pt]{amsart}
\usepackage{geometry}                
\usepackage{graphicx}
\usepackage{amssymb}
\usepackage{epstopdf}
\usepackage{parskip}
\usepackage{mathtools}
\usepackage{txfonts}
\usepackage{url}
\newtheorem{theorem}{Theorem}[section]
\newtheorem{lemma}[theorem]{Lemma}
\newtheorem{proposition}[theorem]{Proposition}
\newtheorem{corollary}[theorem]{Corollary}

\newtheorem{definition}[theorem]{Definition}

\title{Quadratic Equations in Three Variables over Gaussian Integers }
\author{Felix Sidokhine}

\begin{document}
\maketitle

\begin{abstract}
We develop an algebraic method of studying of Diophantine quadratic equations in three variables over the ring of Gaussian integers. 
\end{abstract}

\section{Introduction}
The ring of Gaussian integers is a unique factorization domain and even Euclidean, however, it is extremely difficult to reproduce any of the arguments used in the case of rational integers. Many of laminations come at least from two places: the fundamental theorem of arithmetic and the lack of well - ordering of $\mathbb{Z}[i]$. Hence that for Diophantine equations over Gaussian integers should develop its own unique methods of the study maximizing algebraic properties of the ring $\mathbb{Z}[i]$.

\section{The Ring of Gaussian Integers }

\subsection{The Fundamental Domain}

\begin{definition}
A set $D=\{z \in \mathbb{Z}[i] | \frac{-\pi}{4} < Arg(z) \leq \frac{\pi}{4} \}$ of $\mathbb{Z}[i]$  will call a fundamental domain.
\end{definition}

\begin{proposition}
Any element z $\in$ $\mathbb{Z}[i]$   can be represented as $z = i^n z'$, where $z' \in D$, $0 \leq n \leq 3$.
\end{proposition}

\subsection{The Fundamental Theorem of Arithmetic}

\begin{theorem}
$p = R(p) + iI(p)$ is Gaussian prime if and only if it satisfies one of conditions: 
\begin{itemize}
\item $N(p) = R(p)^2  + I(p)^2$  is a prime in  $\mathbb{Z}$
\item $p$  is a prime in $\mathbb{Z}$ and $p \equiv 3 \mod 4$.
\end{itemize}
\end{theorem}

\begin{definition}
Let $\pi([i])$  denote a set of all primes of the ring of Gaussian integers.
\end{definition}

Our first challenge is to eliminate the associated with the presence of the unit group $U$ of the ring Gaussian integers. To do so, we first claim that any $z$ can be expressed as: $z = i^n z'$ where the argument of $z'$ satisfies a condition $\frac{-\pi}{4} < Arg(z') \leq \frac{\pi}{4}$. One can verify this statement directly. In turn we can claim that the zone $(\frac{-\pi}{4}, \frac{\pi}{4}]$ contains the set of primes $\pi' \subset \pi([i])$ which express uniquely any $z=i^n p_1^{\alpha_1} p_2^{\alpha_2}...p_m^{\alpha_m}$ where $p_i \in \pi'$ and all $p_i$ are distinct primes. This ``uniqueness'' is no longer the uniqueness as understood in algebraic number theory \cite{Neukirch:1999aa} but perfect copy of the uniqueness of factorization that hold in $\mathbb{Z}$ (see \cite{Sidokhine:2013aa}).

\begin{definition}
The set $\pi'$ we shall call the set of canonical prime elements of Gaussian integers.
\end{definition}

\begin{theorem}
Let an element z be a nonzero no unit of Gaussian integers.Then an irreducible factorization $z$, $z=i^n p_1^{\alpha_1} p_2^{\alpha_2}...p_m^{\alpha_m}$, $0 \leq n \leq 3$,  $p_i \in \pi'$, and all $p_i$ are distinct, and this representation is unique up to permutations.
\end{theorem}

\subsection{Multiplicative Properties}

\begin{definition}
Gaussian integer $\alpha$ is called even if $\alpha \equiv 0 \mod (1+ i)$. 
\end{definition}

\begin{proposition}
Gaussian integer $\alpha$ is even if and only if the real $R(\alpha)$ and imaginary $I(\alpha)$ parts of $\alpha$ satisfy the condition  $R(\alpha) + I(\alpha) \equiv 0 \mod 2$.
\end{proposition}

Let $E = \{\alpha \in \mathbb{Z}[i] | R(\alpha) + I(\alpha) \equiv 0 \mod 2\}$ be the set of all even Gaussian integers.

\begin{theorem}
Let subsets $E_0=\{\alpha \in E | R(\alpha) \equiv 0 \mod 2\}$ and $E_I=\{\alpha \in E | R(\alpha)\equiv 1 \mod 2\}$ then $E=E_0 \cup E_I$  and $E_0 \cap E_I= \emptyset$.
\end{theorem}

\begin{theorem}
The subsets $E_0$ and $E_I$ of the set $E$ have the following multiplicative properties
$E_0  \cdot E_0   \subset E_0$  ; $E_0 \cdot E_I  \subset E_0$ ;  $E_I \cdot E_I  \subset E_0$.
\end{theorem}

\begin{lemma}\label{lem1}
Let Gaussian integer $\alpha$ belong to E. Then
\begin{itemize}
\item if $\alpha \in E_0$ then $R(\alpha^2) \equiv 0 \mod 4$ and $I(\alpha^2) \equiv 0 \mod 4$
\item if $\alpha \in E_I$  then $R(\alpha^2) \equiv 0 \mod 4$ and $I(\alpha^2) \equiv 2 \mod 4$
 \end{itemize}
\end{lemma}

\begin{definition}
Gaussian integer $\alpha$ is called odd if  $\alpha \nequiv 0 \mod (1 + i)$.
\end{definition}

\begin{proposition}
Gaussian integer $\alpha$ is odd if and only if $R(\alpha)+I(\alpha)\equiv1 \mod 2$. 
\end{proposition}

Let $O=\{\alpha\in \mathbb{Z}[i] | R(\alpha)+I(\alpha)\equiv1 \mod 2\}$ is the set of all odd Gaussian integers.

\begin{theorem}
Let subsets $O_0=\{ \alpha \in O | R(\alpha) \equiv 0 \mod 2\}$ and $O_I=\{ \alpha \in O | R(\alpha) \equiv 1 \mod 2\}$ then  $O=O_0\cup O_I$  and $O_0 \cap O_I = \emptyset$.
\end{theorem}

\begin{theorem}
The subsets $O_0$ and $O_I$ of the set $O$ have the following multiplicative properties
$O_0 \cdot O_0  = O_I$; $O_0 \cdot O_I  = O_0$; $O_I \cdot O_I  = O_I$.
\end{theorem}

\begin{corollary}
The set $O_I=\{ \alpha \in O | R(\alpha) \equiv 1 \mod 2\}$  is a multiplicative monoid. 
\end{corollary}

\begin{theorem}
The set $O^I=\{\alpha \in O | R(\alpha) \equiv 1 \mod 4\}$  is a multiplicative submonoid of $O_I$.
\end{theorem}

\begin{corollary}
Any element $\alpha \in O$ can be represented a unique way $\alpha=i^n \beta$ where $\beta\in O^I$.
\end{corollary}

\begin{lemma}\label{lem2}
Let Gaussian integer $\alpha$ belong to $O$. Then
\begin{itemize} 
\item if  $\alpha \in O_0$ then $R(\alpha^2)\equiv -1 \mod 4$ and $I(\alpha^2) \equiv 0 \mod 4$
\item if  $\alpha \in O_I$  then $R(\alpha^2)\equiv   1 \mod 4$ and $I(\alpha^2) \equiv 0 \mod 4$
\end{itemize}
\end{lemma}

\begin{theorem}\label{thm9}
The set $O$ and the subsets $E_0, E_I \subset E$ have the following multiplicative properties  
$O \cdot E_0 = E_0$; $O \cdot E_I  = E_I$.
\end{theorem}

Let $E_0'=\{\alpha \in E_0|\gcd(R(\alpha), I(\alpha))=2\}$, $E_0''=\{\alpha \in E_0 | \alpha=(1+i)^2 \beta, \beta \in O\}$ be subset of $E_0$ then

\begin{theorem}
The subsets $E_I$, $E_0''$, $E_0'$ of the set $E$ have the following multiplicative properties 
$E_I \cdot E_0'' \subset E_0'$ ; $E_I \cdot E_I  \subset E_0'$.
\end{theorem}

\subsection{The Fundamental Theorem of Arithmetic, a Special Form}

\begin{theorem}\label{thm11}
Let $z$ be a nonzero non-unit of $\mathbb{Z}[i]$. The irreducible factorization $z$ can be wrote as  $z=i^\alpha (1+i)^{\alpha_1} p_2^{\alpha_2}...p_m^{\alpha_m}$, where $p_i \in O^I (\pi') = \{ q \in \pi' |  p=i^k q, 0 \leq k \leq 3 \text{ and }  R(p)\equiv 1 \mod 4\}$.
\end{theorem}

\begin{definition}[Great Common Divisor]
Let $z=i^\alpha (1+i)^{\alpha_1} p_2^{\alpha_2}...p_m^{\alpha_m}$, $z'=i^\beta (1+i)^{\beta_1} p_2^{\beta_2}...p_m^{\beta_m}$, where all $p_i$  are the distinct primes and belong to $O^I (\pi')$, then $\gcd(z, z')=i^\gamma (1+i)^{\gamma_1} p_2^{\gamma_2}...p_m^{\gamma_m}$, where  $\gamma=\min(\alpha, \beta)$, $\gamma_1=\min(\alpha_1, \beta_1),..., \gamma_m=\min(\alpha_m, \beta_m)$. 
\end{definition}

\section{Mordell's Lemma}

\begin{definition}
A set $G=\{ z \in \mathbb{Z}[i] | z = (1+i)^{\alpha_1} p_2^{\alpha_2}...p_m^{\alpha_m}$, where all $p_i$  are distinct primes and belong to $O^I (\pi')$ we will call $G$ - set.
\end{definition}

\begin{lemma}\label{lemMord}
Let us consider the following equation over $\mathbb{Z}[i]$, where $k$ and $V$ are constants subject to $kV \neq 0$, $k, V \in G$ and $\gcd(k,V) = 1$:

\begin{equation*}
\begin{cases}
XY=kV^2 \\
\gcd(X, Y) \in U
\end{cases}
\end{equation*}

Then every solution $X, Y$ of the equation over $\mathbb{Z}[i]$ has the form  $X=i^t k_1 P^2,  Y=i^{-t} k_2 Q^2$, where $k=k_1 k_2$, $V=PQ$, $\gcd (k_1 P, k_2 Q)=1$  and  $k_1, k_2, P, Q \in G$,  $t$ ( $0 \leq t \leq 3$) is an integer.
\end{lemma}

\begin{proof}
Let $X=i^t k_1 P^2$,  $Y=i^{-t} k_2 Q^2$ then $X, Y$ is a solution of the equation.

Let $X_0, Y_0$ be a solution of the equation and $\gcd(X_0,  Y_0) \in U$. Since $\gcd(k, V) = 1$ and $X_0 Y_0=kV^2$ then by theorem \ref{thm11}, $X_0=i^t k_1 P^2$,  $Y_0=i^s k_2 Q^2$, and since $X_0 Y_0=i^{t+s} k_1 P^2 k_2 Q^2 =i^{t+s} kV^2$  then we have $t + s = 0$ and $s = - t$.
\end{proof}

\begin{corollary}
Let $k, V$ belong to $G$ and $\gcd(k, V) = 1$. Let $X_0,Y_0$ be a solution of the equation 
$XY=kV^2, \gcd(X, Y) \in U$.
Then there is such an integer $n$, $0 \leq n \leq 3$  that  $i^n X_0, i^{-n} Y_0 \in G$  is a solution of given equation.
\end{corollary}

\section{Quadratic Equations in Three Variables over Gaussian Integers}

\subsection{The Equation $X^2+Y^2+Z^2=0$}

Let $S$ be a set of solutions of the equation $X^2+Y^2+Z^2=0$, $(X, Y, Z) \in \mathbb{Z}[i]^3$, $XYZ\neq 0$ and $\gcd(X, Y, Z) \in U$.

 \subsubsection{A canonical form of the equation $X^2+Y^2+Z^2=0$ in special variables}

\begin{lemma}
Let $(\alpha, \beta, \gamma)$ belong to $S$ then the product $\alpha\beta\gamma \equiv 0 \mod (1+ i)^2$. 
\end{lemma}

\begin{proof}
Let $\alpha^2+\beta^2+\gamma^2=0$ and $\alpha\beta\gamma \nequiv 0 \mod (1 + i)$ so $\alpha, \beta, \gamma \in O$. Then $R(\alpha^2+\beta^2+\gamma^2) = R(\alpha^2) + R (\beta^2) + R (\gamma^2) = 0$.
According to lemma \ref{lem2}, $R(\alpha^2) + R (\beta^2) + R (\gamma^2)\equiv \pm 1 \mod 4$.That is a contradiction. 

Let $\beta$ be even and $\beta = (1+ i)\beta'$ where $\beta' \in O$. Then $I(\alpha^2+\beta^2+\gamma^2) = I(\alpha^2)+I(\beta^2)+I(\gamma^2)=0$. According to theorem 2.20 and lemmas \ref{lem1}, \ref{lem2} $I(\alpha^2)+I(\beta^2)+I(\gamma^2) \equiv 2 \mod 4$. That is a contradiction
\end{proof}

\begin{lemma}
Let $(\alpha, \beta, \gamma)$ belong to $S$. Since $\alpha= i^m \alpha'$, $\beta = i^n (1+ i)^{2+k}\beta'$, $\gamma = i^l \gamma'$ where $m, n, l$ take values $0,1$ and $\alpha', \beta', \gamma'$ belong to $O^I$ then $m + l \equiv 1 \mod 2$.
\end{lemma}

\begin{proof}
According to theorem \ref{thm11}, the presentation $\alpha= i^m \alpha'$, $\beta = i^n (1+ i)^{2+k} \beta'$, $\gamma = i^l \gamma'$, where $\alpha', \beta', \gamma'$ belong to $O^I$, is uniquely. Let $m+l \equiv 0 \mod 2$ then we have $\alpha'^2\pm\beta^2+\gamma'^2=0$. Further $R(\alpha'^2\pm\beta^2+\gamma'^2)= R(\alpha'^2) \pm R (\beta^2) + R (\gamma'^2) = 0$. According to lemmas \ref{lem1}, \ref{lem2} and $R(\alpha'^2) \pm R (\beta^2) + R (\gamma'^2)\equiv 2 \mod 4$. We have contradiction.
\end{proof}

\begin{theorem}
The canonical form of the equation $X^2+Y^2+Z^2=0$ has the form $X^2+Y^2=Z^2$ if we take the values $X, Z$ from $O^I$ and $Y=(1+i)^{2+\gamma_1} p_2^{\gamma_2}...p_m^{\gamma_m}$, and all $p_i \in O^I$ are distinct primes. Any solution of the equation $X^2+Y^2+Z^2=0$  with conditions $\gcd(X, Y) = 1$, $XYZ\neq 0$ can be reduced to a solution of the equation $X^2+Y^2=Z^2$ with the same conditions and inversely.
\end{theorem}

\begin{proof}
Let $(\alpha, \beta, \gamma)$ belong to $S$, $\alpha=i^m \alpha'$, $\beta=i^n (1+i)^{2 +k} \beta'$, $\gamma=i^l \gamma'$, where $m, n, l$ take values $0, 1$ and $\alpha', \beta', \gamma' \in O^I$. Since $(\alpha, \beta, \gamma)$ is a solution so $(i^{m-n} \alpha', (1+ i)^{2+k} \beta', i^{l-n} \gamma')$ where $m - n$, $l - n$ take values $0, 1$ is also a solution of the equation $X^2+Y^2+Z^2=0$. According lemma 4.2, $m+l \equiv 1 \mod 2$ so either $(\alpha', (1+i)^{2+k} \beta', i\gamma')$ or $(i\alpha', (1+i)^{2+k} \beta', \gamma')$ is a solution of the equation $X^2+Y^2+Z^2=0$  and $(\alpha', (1+i)^{2+k} \beta', \gamma')$  or $(\gamma', (1+i)^{2+k} \beta', \alpha')$ is a solution of the equation $X^2+Y^2=Z^2$. Thus $X = \alpha', Z = \gamma'$ and  $Y=(1+i)^{2+k} \beta'$. Inversely it is true.
\end{proof}

\subsubsection{The solutions of the equation $X^2+Y^2+Z^2=0$ in special variables}

\begin{theorem}
A solution of the equation $X^2+Y^2=Z^2$, where $\gcd(X, Y) \in U$ and $XYZ \neq 0$, can be represented in the parametric form: $X = i^{t+1} (P^2-(-1)^t Q^2), Y=(1+i)^2 PQ, Z=i^{t+1} (P^2+(-1)^t Q^2)$
where $\gcd(P, Q) = 1$, $PQ \equiv 0 \mod (1+i)$, $0 \leq t \leq 3$. $P, Q$ are the products of the factors $(1+i)^{\gamma_1}, p_2^{\gamma_2},...,p_m^{\gamma_m}$, and all $p_i \in O^I$ are distinct primes.  
\end{theorem}

\begin{proof}
Let $(\alpha, \beta, \gamma)$ be a solution of the equation. Thus we have the equality $\alpha^2+\beta^2=\gamma^2$ where $\alpha, \gamma \in O^I$  and $\beta=(1+i)^{2+\gamma_1} p_2^{\gamma_2}...p_m^{\gamma_m}$,   $p_i \in O^I$ and all $p_i$  are distinct. This leads to: $xy = z^2$, $\gcd(x, y)\in U$,
where $x=\frac{\gamma-\alpha}{(1+i)^2}$ , $y=\frac{\gamma+\alpha}{(1+i)^2}$, $z=\frac{\beta}{(1+i)^2}$. According to lemma \ref{lemMord} there exists an integer $t$ such that $x'=\frac{\gamma-\alpha}{i^{-t} (1+i)^2}$ , $y'=\frac{\gamma+\alpha}{i^t (1+i)^2}$ are faithful squares belonging to $G$, $x'y'= xy$, and $\frac{\gamma-\alpha}{i^{-t} (1+i)^2} =Q^2$, $\frac{\gamma+\alpha}{i^t (1+i)^2}  =P^2$.
$\alpha= i^{t+1} (P^2-(-1)^t Q^2)$, $\beta=(1+i)^2 PQ$, $\gamma=i^{t+1} (P^2+(-1)^t Q^2)$ 
where $\gcd(P, Q) =1$, $PQ \equiv 0 \mod(1+ i)$, $0\leq t \leq 3$. $P, Q$ are the products of the factors $(1+i)^{\gamma_1},p_2^{\gamma_2},...,p_m^{\gamma_m}$, and all $p_i \in O^I$ are distinct primes.
\end{proof}

\subsection{The Equation  $X^2+iY^2+Z^2=0$} 

Let $S$ be a set of solutions of the equation $X^2+iY^2+Z^2=0$, where $(X, Y, Z) \in \mathbb{Z}[i]^3$,  $XYZ \neq 0$ and $\gcd(X, Y, Z) \in U$. 

\subsubsection{Canonical forms of the equation   $X^2+iY^2+Z^2=0$ in special variables}

\begin{lemma}
Let $(\alpha, \beta, \gamma)$ belong to $S$ then $\alpha\beta\gamma \equiv 0 \mod (1+ i)$.
\end{lemma}

\begin{proof}
Let $\alpha^2+i\beta^2+\gamma^2=0$ and $\alpha\beta\gamma \nequiv 0 \mod (1 + i)$ so $\alpha, \beta, \gamma \in O$. Then $I(\alpha^2+i\beta^2+\gamma^2)=I(\alpha^2)+I(i\beta^2)+I(\gamma^2)=I(\alpha^2)+R(\beta^2)+I(\gamma^2)=0$. According to lemma \ref{lem2} we have $I(\alpha^2)+R(\beta^2)+I(\gamma^2)\equiv\pm1 \mod 4$  that is a contradiction.
Let us show that $\beta \equiv 0 \mod (1+ i)$. Let $\alpha\equiv 0 \mod (1+ i)$  and $\beta, \gamma \nequiv 0 \mod (1+ i)$ then according to lemmas \ref{lem1}, \ref{lem2} we have $I(\alpha^2)+R(\beta^2)+I(\gamma^2)\equiv1 \mod 2$. That is a contradiction.
\end{proof}

For $\beta$ there is two cases. 
\begin{itemize}
\item Case 1: $\beta \equiv 0 \mod (1+ i)$; $\beta \nequiv 0 \mod(1+i)^2$. 
\item Case 2: $\beta\equiv 0 \mod(1+i)^2$.
\end{itemize}

Case 1: Let $\beta \equiv 0 \mod (1+ i)$ and $\beta\nequiv 0 \mod(1+i)^2$ then

\begin{lemma}
Let $(\alpha, \beta, \gamma)$ belong to $S$. Since $\alpha=i^m \alpha'$, $\beta=i^n (1+i)\beta'$, $\gamma=i^l \gamma'$, where $m, n, l$ take values $0, 1$ and $\alpha', \beta', \gamma'$ belong to $O^I$ then $m + l \equiv 0 \mod 2$.
\end{lemma}

\begin{proof}
According to theorem \ref{thm11}, $\alpha=i^m \alpha'$, $\beta=i^n (1+i)\beta'$, $\gamma=i^l \gamma'$, where $\alpha', \beta', \gamma' \in O^I$, is uniquely. Let $n + m \equiv 1 \mod 2$ then $\alpha'^2\pm i\beta^2+\gamma'^2=0$  and $R(\alpha'^2\pm i \beta^2+\gamma'^2)=R(\alpha'^2) \pm I(\beta^2) - R(\gamma'^2)=0$.  According to theorem 2.20, lemmas \ref{lem1}, \ref{lem2}  $R(\alpha'^2) \pm I(\beta^2)-R(\gamma'^2)\equiv2 \mod 4$. That is a contradiction.
\end{proof}

\begin{theorem}
The canonical representation of the equation  $X^2+iY^2+Z^2=0$ has the form $X^2+Z^2=\pm iY^2$ if we take the values $X, Z$ from $O^I$ and $Y=(1+i)  p_2^{\gamma_2}...p_m^{\gamma_m}$ , where  $p_i \in O^I$ and all  $p_i$ are distinct primes. Any solution of the equation $X^2+iY^2+Z^2=0$, with conditions $\gcd(X, Y)=1$, $XYZ\neq 0$ can be reduced to a solution of the equation $X^2+Z^2=\pm iY^2$, with the same conditions and inversely.
\end{theorem}

\begin{proof}
Let $(\alpha, \beta, \gamma)$ belong to $S$ and $\alpha=i^m \alpha'$, $\beta=i^n (1+i)\beta'$, $\gamma=i^l \gamma'$, where $m, n, l$ take values $0, 1$ and $\alpha', \beta', \gamma' \in O^I$. Since $(\alpha, \beta, \gamma)$ is a solution so we can consider that $(i^{m-n} \alpha', (1+i)\beta', i^{l-n} \gamma')$ where $m - n, l - n$ take values $0, 1$ is also a solution of the equation $X^2+iY^2+Z^2=0$. According lemma 4.6, $m + l \equiv 0 \mod 2$ hence $(\alpha', (1+i)\beta',\gamma')$ or $(i\alpha', (1+i)\beta',i\gamma')$ is a solution of the equation $X^2+iY^2+Z^2=0$ and $(\alpha', (1+i)\beta',\gamma')$ is a solution of the equation $X^2+Z^2=- iY^2$ thus $X = \alpha', Z = \gamma'$ from $O^I$ and $Y = (1+ i)\beta'$ where $\beta' \in O^I$ or $(i\alpha', (1+i)\beta', i\gamma')$ is a solution of the equation $X^2 + iY^2+ Z^2 = 0$ and $(\alpha', (1+ i)\beta', \gamma')$ is a solution of the equation $X^2 + Z^2 = iY^2$ thus $X = \alpha', Z = \gamma'$ from $O^I$ and $Y = (1+ i)\beta'$ where $\beta' \in O^I$. Inversely it is also true.
\end{proof}

Case 2: Let $\beta \equiv 0 \mod(1+i)^2$. Then 

\begin{lemma}
Let $(\alpha, \beta, \gamma)$ belong to $S$. Since $\alpha= i^m \alpha'$, $\beta = i^n (1+ i)^{2+k}  \beta'$, $\gamma = i^l \gamma'$ where $m, n, l$ take values $0, 1$ and $\alpha', \beta', \gamma'$ belong to $O^I$ then $m + l \equiv 1 \mod 2$.
\end{lemma}

\begin{proof}
According to theorem \ref{thm11} the presentation $\alpha= i^m \alpha'$, $\beta = i^n (1+ i)^{2+k}  \beta'$, $\gamma = i^l \gamma'$ where $\alpha', \beta', \gamma' \in O^I$ is uniquely. Let $n + m \equiv 0 \mod 2$ then $\alpha'^2 \pm i\beta^2 + \gamma'^2 = 0$ and $R(\alpha'^2 \pm i\beta^2+\gamma'^2)=R(\alpha'^2)\pm I(\beta^2) +R(\gamma'^2)=0$. According to lemmas \ref{lem1}, \ref{lem2} $R(\alpha'^2)\pm I(\beta^2)+R(\gamma'^2)\equiv 2 \mod 4$. That is a contradiction.
\end{proof}

\begin{theorem}
The canonical representation of the equation $X^2+iY^2+Z^2=0$ has the form $X^2+iY^2=Z^2$ if we take the value $X, Z$ from $O^I$ and $Y=(1+i)^{2+\gamma_1} p_2^{\gamma_2}...p_m^{\gamma_m}$,  $p_i \in O^I$ and all $p_i$ are distinct primes. A solution of the equation $X^2+iY^2+Z^2=0$, with conditions $\gcd(X, Y)=1$, $XYZ\neq 0$ can be reduced to a solution of the equation $X^2+iY^2=Z^2$, with the same conditions and inversely.
\end{theorem}

\begin{proof}
Let $(\alpha, \beta, \gamma)$ belong to $S$ and $\alpha=i^m \alpha'$, $\beta=i^n (1+i)^2 \beta'$, $\gamma=i^l \gamma'$, where $m, n, l$ take values $0, 1$ and $\alpha', \gamma' \in O^I$. Since $(\alpha, \beta, \gamma)$ is a solution so we can consider that $(i^{m-n} \alpha', (1+i)^2 \beta', i^{l-n} \gamma')$, where $m - n$, $l - n$ take values $0, 1$ is also a solution of the equation $X^2+iY^2+Z^2=0$. According lemma 4.8, $m + l \equiv 1 \mod 2$ so $(\alpha', (1+ i)^2 \beta', i\gamma')$ or $(i\alpha', (1+ i)^2 \beta', \gamma')$ is a solution of the equation $X^2+iY^2+Z^2=0$ and $(\alpha', (1+ i)^2 \beta',\gamma')$ is a solution of the equation $X^2+iY^2=Z^2$ thus $X = \alpha', Z = \gamma'$ from $O^I$ and $Y=(1+i)^2 \beta'$ or $(i\alpha', (1+ i)\beta', \gamma')$ is a solution of the equation $X^2+iY^2+Z^2=0$ and $(\gamma', (1+ i)\beta', \alpha')$ is a solution of $X^2+iY^2=Z^2$ thus $X = \gamma', Z = \alpha'$ from $O^I$ and $Y=(1+i)^2 \beta'$. Inversely it is also true.
\end{proof}

\subsection{The solutions of the equation $X^2+iY^2+Z^2=0$ in special variables}

Case 1: $Y \equiv 0 \mod (1+i)$ and  $Y \nequiv 0 \mod (1+i)^2$.

\begin{theorem}
A solution of the equation $X^2+Z^2=\pm iY^2$, $\gcd(X, Y, Z) \in U$ and $XYZ \neq 0$ can be represented in the  parametric form: $X=i^{t+1} \frac{P^2\pm(-1)^t iQ^2}{1+i}$, $Y=i^t \frac{P^2\mp(-1)^t iQ^2}{1+i}$, $Z=(1 + i)PQ$ where $P, Q$ are the products of the factors $p_2^{\gamma_2},...,p_m^{\gamma_m}$  belonging to $O^I$ and $\gcd(P, Q) = 1$.
\end{theorem}

\begin{proof}
Let $(\alpha, \beta, \gamma)$ be some solution of the equation $X^2 \pm iY^2+Z^2=0$. We have the equality
$\alpha^2+\beta^2=\pm i\gamma^2$ where $\alpha, \beta \in O^I$ and $\gamma=(1+i)p_2^{\gamma_2}...p_m^{\gamma_m}$,  $p_i \in O^I$ and all $p_i$ are distinct. This leads to: $xy = z^2$, $\gcd(x, y) \in U$
where $x=\frac{\alpha+i\beta}{1+i}$, $y=\pm\frac{\alpha-i\beta}{i(1+i)}$, $z=\frac{\gamma}{1+ i}$. According to lemma \ref{lemMord} there exists $t$ $(0\leq t \leq 3)$ such that $x'=\frac{\alpha+i\beta}{i^t (1+i)}$, $y'=\pm\frac{\alpha-i\beta}{i^{1-t} (1+i)}$ are faithful squares belong to $G$, $x'y'= xy$ and $\frac{\alpha+i\beta}{i^t (1+i)}=P^2$,  $\pm\frac{\alpha-i\beta}{i^{1-t} (1+i)} =Q^2$. $\alpha=i^{t+1}\frac{P^2\pm(-1)^t iQ^2}{1+i}$, $\beta=i^t \frac{P^2\mp(-1)^t iQ^2}{1+i}$, $\gamma=(1+i)PQ$  where $P, Q$ are the products of the factors  $p_2^{\gamma_2},...,p_m^{\gamma_m}$ belonging to $O^I$ and $\gcd(P, Q) = 1$.
\end{proof}

Case 2: $Y \equiv 0 \mod (1+i)^2$

\begin{theorem}
A solution of the equation $X^2+iY^2=Z^2$, $\gcd(X, Y) = 1$ and $XYZ \neq 0$  can be represented in the parametric form:  $X=i^{t+1} (P^2-(-1)^t iQ^2)$, $Y=(1+ i)^2 PQ$, $Z=i^{t+1} (P^2+(-1)^t iQ^2)$ 
where $P, Q$ are the products of factors $(1+i)^{\gamma_1}$; $p_2^{\gamma_2}$,...,$p_m^{\gamma_m}$ belonging to $O^I$ and $\gcd(P, Q) = 1$.
\end{theorem}

\begin{proof}
Let $(\alpha, \beta, \gamma)$ be a solution of the equations then we have the following equality
$\alpha^2+i\beta^2=\gamma^2$, where  $\alpha, \gamma \in O^I$  and $\beta=(1+i)^{2+\gamma_1}p_2^{\gamma_2}...p_m^{\gamma_m}$, $p_i \in O^I$ and all $p_i$ are distinct. This leads to: $xy = z^2$, $\gcd(x, y) \in U$, where $x=\frac{\gamma-\alpha}{(1+i)^2}$ ,$ y=\frac{\gamma+\alpha}{(1+i)^2}$ , $z=\frac{\beta}{(1+i)^2}$. According to lemma \ref{lemMord}, there is $t$ such that $x'=\frac{\gamma-\alpha}{i^{1-t} (1+i)^2}$ ,  $y'=\frac{\gamma+\alpha}{i^t (1+i)^2}$  are faithful squares, belong to $G$ and $x'y'= xy$. Thus $\frac{\gamma-\alpha}{i^{1-t} (1+i)^2} = Q^2$, $\frac{\gamma+\alpha}{i^t (1+i)^2} =P^2$. $\alpha=i^{t+1} (P^2-(-1)^t iQ^2)$, $\beta=(1+i)^2 PQ$, $\gamma=i^{t+1} (P^2+(-1)^t iQ^2)$, where $P, Q$ are the products of factors $(1+i)^{\gamma_1},p_2^{\gamma_2},...,p_m^{\gamma_m}$,  $p_i \in O^I$ and $\gcd(P, Q) = 1$.
\end{proof}

\subsection{The Equation $X^2+(1 \pm i)Y^2+Z^2=0$} 

Let $S$ denote a set of solutions of the equation $X^2+(1\pm i)Y^2+Z^2=0$, where $(X, Y, Z)\in\mathbb{Z}[i]^3$, $XYZ \neq 0$ and $\gcd(X, Y, Z) \in U$.
 
\subsubsection{A canonical form of the equation $X^2+(1 \pm i)Y^2+Z^2=0$ in special variables}

\begin{lemma}
Let $(\alpha, \beta, \gamma)$ belong to $S$ then  $\beta\equiv 0 \mod (1+i)^2$. 
\end{lemma}

\begin{proof}
Let $\alpha^2+(1\pm i) \beta^2+\gamma^2=0$ and $\beta\nequiv 0 \mod (1+i)$ so, according to theorem 10, $(1+i) \beta^2$ belongs to $E_I$ and $\alpha, \gamma \in O$. Then $I(\alpha^2+(1\pm i) \beta^2+\gamma^2)=I(\alpha^2)+I((1\pm i)\beta^2)+I(\gamma^2)=0$. Thus $I(\alpha^2)+I((1\pm i)\beta^2)+I(\gamma^2)\equiv0 \mod 2$. According to lemma 2.19, since $(1\pm i) \beta^2 \in E_I$ we have
$I(\alpha^2)+I((1\pm i)\beta^2)+I(\gamma^2)\equiv1 \mod 2$. That is a contradiction. 

Let $\beta \equiv 0 \mod(1+i)$ and $\beta\nequiv 0 \mod (1+i)^2$ then $\beta \in E_I$  and $I((1+i) \beta^2) \equiv 2 \mod 4$. Thus $I(\alpha^2)+I((1\pm i)\beta^2)+I(\gamma^2)\equiv2 \mod 4$. That is a contradiction. Thus $\beta\in E_0, \beta\equiv 0 \mod(1+i)^2$.
\end{proof}

\begin{lemma}\label{lem12}
Let $(\alpha, \beta, \gamma)$ belong to $S$. Since $\alpha=i^m \alpha'$, $\beta=i^n (1+i)^2 \beta'$, $\gamma=i^l \gamma'$, where $m, n, l$ take values $0, 1$ and $\alpha', \beta', \gamma'$ belong to $O^I$ then  $m + l \equiv 1 \mod 2$.
\end{lemma}

\begin{proof}
According to theorem \ref{thm11} the presentation $\alpha=i^m \alpha'$, $\beta=i^n (1+i)^2 \beta'$, $\gamma=i^l \gamma'$, where $\alpha', \gamma'$ belong to $O^I$ is uniquely. Let $m + l \equiv 0 \mod 2$ we have $\alpha'^2\pm(1\pm i)\beta^2+\gamma'^2=0$ then $R(\alpha'^2\pm(1\pm i) \beta^2+\gamma'^2)= R(\alpha'^2) \pm R((1\pm i)\beta^2) + R (\gamma'^2) =0$. According to lemmas \ref{lem1}, \ref{lem2} $R(\alpha'^2)\pm R((1\pm i)\beta^2)+R (\gamma'^2)\equiv2 \mod 4$. That is a contradiction.
\end{proof}

\begin{theorem}\label{thm19}
The canonical representation of the equation $X^2+(1\pm i)Y^2+Z^2=0$ has the form $X^2+(1\pm i)Y^2=Z^2$
if we take the value $X, Z$ from $O^I$ and $Y=(1+i)^{2+\gamma_1} p_2^{\gamma_2}...p_m^{\gamma_m}$,   $p_i\in O^I$ and all $p_i$ are distinct primes. Any solution of the equation  $X^2+(1\pm i)Y^2+Z^2=0$, with conditions $\gcd(X, Y)=1$, $XYZ \neq 0$ can be reduced to a solution of the equation $X^2+(1\pm i)Y^2=Z^2$, with the same conditions and inversely.
\end{theorem}

\begin{proof}
Let $(\alpha, \beta, \gamma)$ belong to $S$ and $\alpha=i^m \alpha'$, $\beta=i^n (1+i)^2 \beta'$, $\gamma=i^l \gamma'$ where $m, n, l$ take values $0, 1$ and $\alpha', \gamma' \in O^I$. Since $(\alpha, \beta, \gamma)$ is a solution so we can consider that $(i^{m-n} \alpha'$, $(1+i)^2 \beta'$, $i^{l-n} \gamma')$, where $m - n$, $l - n$ take values $0, 1$ is also a solution of the equation $X^2+(1\pm i)Y^2+Z^2=0$. According lemma 4.13, $m + l \equiv 1 \mod 2$ so $(\alpha', (1+i)^2 \beta',i\gamma')$ or $(i\alpha', (1+i)^2 \beta',\gamma')$ is a solution of the equation $X^2+(1\pm i)Y^2+Z^2=0$ and $(\alpha', (1+i)^2 \beta',\gamma')$ is a solution of the equation $X^2+(1\pm i)Y^2=Z^2$ thus $X = \alpha'$, $Z = \gamma'$ from $O^I$ and $Y=(1+i)^2 \beta'$ or $(i\alpha', (1+i)^2 \beta',\gamma')$ is a solution of $X^2+(1\pm i)Y^2+Z^2=0$ and $(\gamma', (1+i)^2 \beta',\alpha')$ is a solution of $X^2+(1\pm i)Y^2+Z^2=0$ thus $X = \gamma'$, $Z = \alpha'$ from $O^I$ and $Y=(1+i)^2 \beta'$. Inversely it is also true. 
\end{proof}

\subsubsection{The solutions of the equation $X^2+(1 \pm i)Y^2+Z^2=0$  in special variables}

\begin{theorem}
Any solution of the equation $X^2+(1\pm i)Y^2=Z^2$, $\gcd(X, Y, Z) \in U$ and $XYZ \neq 0$  can be represented  in the parametric from: $X = i^{t+1} (P^2-(-1)^t (1\pm i) Q^2 )$, $Y=(1 +i)^2 PQ$, $Z = i^{t+1} (P^2+(-1)^t (1\pm i) Q^2 )$, where $P, Q$ are the products of the factors $(1+i)^{\gamma_1}, p_2^{\gamma_2},...,p_m^{\gamma_m}$,  $p_i \in O^I$  and all $p_i$ are distinct primes, $\gcd(P, (1+ i)Q) =1$.  
\end{theorem}

\begin{proof}
Let $(\alpha, \beta, \gamma)$ be a solution of the equation. Thus we have the equality
$\alpha^2+(1\pm i) \beta^2=\gamma^2$ where $\alpha, \gamma$ belong $O^I$ and $\beta=(1+i)^{2+\gamma_1} p_2^{\gamma_2}...p_m^{\gamma_m}$,  $p_i \in O^I$  and all $p_i$ are distinct primes $xy=z^2$, $\gcd(x, y)\in U$, where $x=\frac{\gamma-\alpha}{(1\pm i)(1+i)^2}$, $y=\frac{\gamma+\alpha}{(1+i)^2}$ , $z=\frac{\beta}{(1+i)^2}$ . According to lemma \ref{lemMord} there is an integer $t$ such that $x'=\frac{\gamma-\alpha}{i^t (1\pm i)(1+i)^2}$ ,  $y'=\frac{\gamma+\alpha}{i^{-t} (1+i)^2}$  are faithful squares, belong to $G$ and $x'y'= xy$. Thus $\frac{\gamma-\alpha}{i^t (1\pm i)(1+i)^2} =Q^2$, $\frac{\gamma+\alpha}{i^{-t} (1+i)^2} =P^2$ and 
$\alpha= i^{t+1} (P^2-(-1)^t (1\pm i) Q^2)$, $\beta=(1+i)^2 PQ$, $\gamma= i^{t+1} (P^2+(-1)^t (1\pm i) Q^2)$.
where $P, Q$ are the products of the factors $(1+i)^{\gamma_1},p_2^{\gamma_2},...,p_m^{\gamma_m}$,  $p_i\in O^I$ and all $p_i$ are distinct primes and $\gcd(P, (1+ i)Q) = 1$.
\end{proof}

\subsection{The Equation $X^2+iY^2+(1+i)Z^2=0$} Let $S$ denote the set of all solutions of the equation $X^2+iY^2+(1+i)Z^2=0$, where $(X, Y, Z) \in \mathbb{Z}[i]^3$,  $XYZ\neq 0$ and $\gcd(X, Y, Z) \in U$. 

\subsubsection{The canonical form of the equation $X^2+iY^2+(1+i)Z^2=0$ in special variables}

\begin{lemma}
Let $(\alpha, \beta, \gamma)$ belong to $S$ then $\alpha\beta\gamma \nequiv 0 \mod(1+ i)$. 
\end{lemma}

\begin{proof}
Let $\alpha\beta\gamma \equiv 0 \mod(1+i)$. As $\gcd(\alpha, \beta, \gamma) = 1$ so $\gcd(\alpha, \beta) = 1$ and $\alpha\beta \nequiv 0 \mod (1+i)$. Thus $\alpha, \beta$ belong to $O$ then $R(\alpha^2+i\beta^2+(1+i)  \gamma^2)= R(\alpha^2 )+R(i\beta^2 )+R((1+i) \gamma^2 )=R(\alpha^2 )-I(\beta^2 )+R(\gamma^2 )-I(\gamma^2 )=0$ and $R(\alpha^2+i\beta^2+(1+i) \gamma^2)=R(\alpha^2)-I(\beta^2)+R(\gamma^2)-I(\gamma^2)\equiv0 \mod 2$. Since $\gamma\equiv0 \mod (1+i)$  then $R(\alpha^2)\equiv0 \mod 2$ but $\alpha$ belongs to $O$ so $R(\alpha^2)\equiv1 \mod 2$. That is a contradiction. Thus $\gamma$ belongs to $O$.
\end{proof}

\begin{lemma}
Let $(\alpha, \beta, \gamma)$ belong to $S$. Since $\alpha=i^m \alpha'$, $\beta=i^n \beta'$, $\gamma=i^l \gamma'$, where $m, n, l$ take values $0, 1$ and $\alpha', \beta', \gamma'$ belong to $O^I$ then $m+l \equiv n+l \equiv1 \mod 2$.
\end{lemma}

\begin{proof}
According to theorem \ref{thm11}, the presentation $\alpha=i^m \alpha'$, $\beta=i^n \beta'$, $\gamma=i^l \gamma'$, where $\alpha', \beta', \gamma'$ belong to $O^I$ is uniquely. Let $m+l \equiv 0 \mod 2$  we have $\alpha'^2\pm i\beta'^2+(1+i)  \gamma'^2=0$. Hence $R(\alpha'^2\pm i\beta'^2+(1+i) \gamma'^2)=R(\alpha'^2)\mp I(\beta'^2)+R(\gamma'^2)-I(\gamma'^2)\equiv0 \mod 4$.

Therefore, $R(\alpha'^2)+R(\gamma'^2)\equiv0 \mod 4$ but $\alpha', \gamma'\in O^I$ and $R(\alpha'^2)+R(\gamma'^2)\equiv2 \mod 4$. That is a contradiction. 

Let $n+l\equiv0 \mod 2$, $\pm\alpha'^2+i\beta'^2+(1+i)  \gamma'^2=0$ so $I(\pm\alpha'^2+i\beta'^2+(1+i)  \gamma'^2)=\pm I(\alpha'^2)+ R(\beta'^2)+I(\gamma'^2)+R(\gamma'^2) \equiv0 \mod 4$. $R(\beta'^2)+R(\gamma'^2)\equiv0 \mod 4$ but $\beta', \gamma' \in O^I$ and $R(\beta'^2)+R(\gamma'^2)\equiv2 \mod 4$. That is a contradiction. 
\end{proof}

\begin{theorem}
The canonical representation of the equation $X^2+iY^2+(1+i)Z^2=0$ has the form $X^2+iY^2=(1+i)Z^2$ if we take the values $X, Y, Z$ from $O^I$. Any solution of the equation $X^2+iY^2+(1+i)Z^2=0$, with conditions $\gcd(X, Y)=1$, $XYZ\neq0$ can be reduced to a solution of the equation $X^2+iY^2=(1+i)Z^2$, with the same conditions and inversely.
\end{theorem}

\begin{proof}
Let $(\alpha, \beta, \gamma)$ belong to $S$ and $\alpha=i^m \alpha'$, $\beta=i^n \beta'$, $\gamma=i^l \gamma'$, where $m, n, l$ take values $0, 1$ and $\alpha', \beta',\gamma' \in O^I$. As $(\alpha, \beta, \gamma)$ is a solution so we can consider that $(i^{m-l} \alpha',i^{n-l} \beta ', \gamma')$, where $m - l$, $n - l$ take the value $1$ is also a solution of the equation $X^2+iY^2+(1+i)Z^2=0$. According lemma 4.17, $m + l \equiv 1 \mod 2$ and $n + l \equiv 1 \mod 2$ so $(i\alpha', i\beta', \gamma')$ or $(\alpha', \beta', i\gamma')$ is a solution of the equation $X^2+iY^2+(1+i)Z^2=0$ and $(\alpha', \beta', \gamma')$ is a solution of $X^2+iY^2=(1+i)Z^2$  where $X = \alpha', Y = \beta', Z = \gamma'$ from $O^I$. Inversely it is also true.
\end{proof} 
 
 \subsubsection{The solutions of the equation $X^2+iY^2+(1+i)Z^2=0$  in special variables}
 
 \begin{theorem} 
Some solutions of the equation $X^2+iY^2=(1+i)Z^2$, $\gcd(X, Y) \in U$, $XYZ \neq 0$ can be represented in the  parametric form: $X=P^2-(1+i)iQ^2$, $Y=P^2-2(1+i)PQ+(1+ i)iQ^2$, $Z=P^2-2iPQ+(1+i)iQ^2$
where $P, Q$ are products of factors $(1+i)^{\gamma_1},p_2^{\gamma_2},...,p_m^{\gamma_m}$, $p_i \in O^I$ and all $p_i$ are distinct primes, $\gcd(P,Q) = 1$, $Q \equiv 0 \mod (1+i)$.
\end{theorem}

\begin{proposition}
The canonical equation $X^2+iY^2=(1+i)Z^2$, with conditions $\gcd(X, Y) \in U$, $XYZ \neq 0$
has a solution if and only if the system of equations 
\begin{equation*}
\begin{cases}
x - y = u + v \\ 
xy = iuv,
\end{cases}
\end{equation*}
where $x + y, u+v, u - v \in O^I$, $uv \equiv 0 \mod (1+i)$ and $\gcd(x, y)$, $\gcd(u, v) \in U$ has a solution.
\end{proposition}

\begin{proof}
Let $(\alpha, \beta, \gamma)$, where $\alpha, \beta, \gamma$ belong $O^I$ and $\gcd(\alpha, \beta, \gamma) = 1$, be a solution of the equation, i.e. $\alpha^2+i\beta^2=(1+i) \gamma^2$. This leads to the following equality $\alpha^2-\gamma^2=i (\gamma^2-\beta^2)$.

Let $m=\frac{\alpha + \gamma}{2}$, $n=\frac{\alpha - \gamma}{2}$, $k=\frac{\gamma + \beta}{2}$,$l= \frac{\gamma - \beta}{2}$,  where $m, n, k, l$ are Gaussian integers, $n \equiv l \equiv 0 \mod (1+ i)$, $\gcd(m, n), gcd(k, l) \in U$. Then $(m, n, k, l)$ is a solution of the system of equations $x - y = u + v, xy = i uv$, where $x + y, u+ v, u - v \in O^I$, $uv \equiv 0 \mod (1+i)$ and $\gcd (x, y), gcd (u, v) \in U$. 

Let $(m, n, k, l)$ is a solution of the system of equations $x - y = u + v, xy = i uv$ where $x + y, x - y, u - v \in O^I$ and $\gcd(x, y), \gcd(u, v) \in U$. Let $m + n = \alpha, k - l = \beta, m - n = k + l = \gamma$ where $\alpha, \beta, \gamma \in O^I$ and $\gcd(\alpha, \gamma) = \gcd (\beta, \gamma) = 1$ then $m=\frac{\alpha + \gamma}{2}, n=\frac{\alpha - \gamma}{2}, k=\frac{\gamma + \beta}{2},l=\frac{\gamma - \beta}{2}$.  Since $mn=i kl$ then $\alpha^2+i\beta^2=(1+i)  \gamma^2$ and $(\alpha, \beta, \gamma)$ is a solution of $X^2+iY^2=(1+i)Z^2$, $\gcd(X, Y)=1$, $XYZ \neq 0$.
\end{proof}

\begin{proposition}
The system of equations 
\begin{equation*}
\begin{cases}
x - y = u + v \\ 
xy = i uv
\end{cases}
\end{equation*}
where $\gcd(x, y), \gcd(u, v) \in U$, $uv \equiv 0 \mod (1+i)$ has a solution over Gaussian integers if and only if the quadratic equation 
$z^2  - (u + v)z - iuv = 0$ has a solution over $\mathbb{Z}[i]$.
\end{proposition}

\begin{proof}
Let $(x_0, y_0,u_0  , v_0)$ be a solution of the system equations then $x_0$ is a solution of the equation 
$z^2  - (u_0  + v_0)z - iu_0 v_0= 0$.

Let $u_0$, $v_0$ such that the equation $z^2  - (u_0  + v_0)z - iu_0 v_0= 0$ has the solutions over $\mathbb{Z}[i]$ then $z_1=\frac{ u_0 + v_0 + \sqrt{(u_0 + v_0)^2 + 4iu_0v_0}}{2}, z_2=\frac{ u_0 + v_0 - \sqrt{(u_0 + v_0)^2 + 4iu_0v_0}}{2}$, $z_1z_2 = -iu_0v_0$.

Then $(x_0 = z_1, y_0 = - z_2, u_0, v_0)$ is a solution of the system of equations $x - y = u + v, xy = iuv$, where $\gcd(x, y)$ belong to $U$.
\end{proof}

\begin{corollary}
Let a discriminant $D$ of a quadratic equation 
$z^2  - (u + v)z - iuv = 0$, where $u + v \in O^I$, $uv \equiv 0 \mod (1+i)$ and $\gcd(u, v) \in U$, be a square in $\mathbb{Z}[i]$ then the system of equations
\begin{equation*}
\begin{cases} 
x - y = u + v \\
xy = iuv
\end{cases}
\end{equation*}
where $\gcd(x, y) \in U$, has a solution.
\end{corollary}

\begin{proof}
A discriminant of the quadratic equation has the form $D^2=(u+v)^2+4iuv$.  Since $\gamma = u + v$ belongs to $O^I$ and $D^2 \equiv \gamma^2  \mod (1+i)^4$ hence $D = \pm D'$, where $D' \in O^I$. Thus $x=\frac{\gamma+D'}{2}$, $y=\frac{-\gamma+D'}{2}$, where $x, y$ are Gaussian integers and $\gcd(x, y) \in U$.
\end{proof}

\begin{proposition}
The solutions of the equation $D^2=(u+v)^2+4iuv$, where $D$ and $u + v \in O^I$, $\gcd(u, v) \in U$ and $v\equiv 0 \mod(1+i)$, have the parametric form: $D = P^2 + (1-i) Q^2$, $u = P^2 -(1- i) Q^2-(2i+1)PQ$ and $v=PQ$ where $P, Q$ are the products of factors $(1+i)^{\gamma_1},p_2^{\gamma_2},...,p_m^{\gamma_m}$, $p_i \in O^I$ and all $p_i$  are distinct primes, $\gcd(P, Q) =1$, $Q \equiv 0 \mod (1+i)$.
\end{proposition}

\begin{proof}
Let the equation $D^2=(u+v)^2+4iuv$  have a solution then the equation  $X^2+(1-i)Y^2=Z^2$, $\gcd(X, Y,Z)=1$, $XYZ \neq 0$ also has a solution $X = D$, $Y=(1+i)^2 v$, $Z = u + (2i+1)v$. According to theorem 4.15, the solution of the equation has a form $X = i^{1+t}(P^2-(-1)^t (1-i) Q^2)$, $Y = (1+ i)^2 PQ$, $Z = i^{1+t}(P^2+(-1)^t (1-i) Q^2)$ where $P, Q$ are the products of the factors $(1+i)^{\gamma_1},p_2^{\gamma_2},...,p_m^{\gamma_m}$,   $p_i \in O^I$ and all $p_i$  are distinct primes, $\gcd(P, Q) = 1$, $Q \equiv 0 \mod (1+i)$. Thus the solution of the initial equation has the form $D = i^{1+t}(P^2-(-1)^t (1-i)Q^2), u = i^{1+t} (P^2+(-1)^t (1-i)Q^2)-(2i+1)PQ, v=PQ$. Since $D$ and $P$ belong to $O^I$ then $t=3$. Thus we have $D=P^2 + (1-i)Q^2$, $u = P^2 - (1-i)Q^2 - (2i + 1)PQ$, $v = PQ$.
\end{proof}

\begin{corollary}
A parametric presentation of the solutions of the system of equations 
\begin{equation*}
\begin{cases}
x - y = u + v \\ 
xy = i uv
\end{cases}
\end{equation*}
where $\gcd(x, y), gcd(u, v) \in U$ has the form $x = P^2- iPQ$, $u=P^2 - (1-i) Q^2-(2i+1)PQ$,  $y = (1-i) Q^2+iPQ$,  $v = PQ$ where $P, Q$ are the products of the factors $(1+i)^{\gamma_1},p_2^{\gamma_2},...,p_m^{\gamma_m}$,   $p_i \in O^I$ and all $p_i$  are distinct primes, $\gcd(P, Q) = 1$, $Q \equiv 0 \mod (1+i)$.
\end{corollary}

\begin{proof}
Corollary 4.24 are direct consequence of proposition 4.23 and corollary 4.22.
\end{proof}

\newpage

\begin{proof}[Proof of theorem 4.19]
Theorem 4.19 is a direct consequence of corollary 4.24 and proposition 4.20.
\end{proof}

\section{Conclusion}

In the work we have developed one of the feasible algebraic methods of studying of Diophantine quadratic equations maximizing algebraic properties of the ring $\mathbb{Z}[i]$.

Our approach differs from others similar (see \cite{Sexauer:1966aa,Kubota:1972aa}) due to using the fundamental theorem arithmetic \cite{Sierpinski:1988aa} and Mordell's lemma \cite{Mordell:1969aa} are adapted to solving of Diophantine quadratic equations in three variables over the ring of Gaussian integers.

We intend to apply the method offered and results obtained to studying of the quartic equations $aX^4+bX^2$ $Y^2+cY^4=dZ^2$ over the ring $\mathbb{Z}[i]$, including as new equations $X^4+6X^2 Y^2+Y^4=Z^2$ (Mordell's equation over the ring $\mathbb{Z}[i]$), $X^4+6(1+i)X^2 Y^2+2iY^4=Z^2$ and $X^4\pm Y^4=(1+i)Z^2$ as well as the equations $X^4+Y^4=Z^2$ (Fermat - Hilbert) and $X^4\pm Y^4=iZ^2$ (Szab{\'o} - Najman).

\bibliographystyle{IEEEtran}
\bibliography{references}

\end{document}